\documentclass[11pt]{article}

\usepackage[T1]{fontenc}
\usepackage[utf8]{inputenc}
\usepackage{amsmath,amssymb,amsthm,mathtools}
\usepackage{newtxtext,newtxmath}
\usepackage[a4paper,textwidth=15.2cm,textheight=22.8cm,centering]{geometry}
\usepackage{microtype}
\usepackage[hidelinks]{hyperref}
\usepackage{enumitem}
\usepackage{mathrsfs}
\usepackage{titling}
\usepackage{authblk}

\newtheorem{theorem}{Theorem}[section]
\newtheorem{proposition}[theorem]{Proposition}
\newtheorem{lemma}[theorem]{Lemma}
\newtheorem{corollary}[theorem]{Corollary}
\theoremstyle{remark}
\newtheorem{remark}[theorem]{Remark}
\numberwithin{equation}{section}

\newcommand{\R}{\mathbb R}
\newcommand{\C}{\mathbb C}

\newcommand{\norm}[1]{\left\|#1\right\|}

\newcommand{\conv}{\operatorname{conv}}
\newcommand{\comp}{\operatorname{comp}}
\newcommand{\caplog}{\operatorname{cap}}
\newcommand{\diam}{\operatorname{diam}}
\newcommand{\dist}{\operatorname{dist}}
\newcommand{\doi}[1]{\href{https://doi.org/#1}{doi:#1}}

\title{Optimal polynomial meshes \\ beyond geometric boundary regularity \\ via Green sublevels}

\author[1,2]{Damián Pinasco}
\author[1]{María Victoria Venuti}
\affil[1]{Universidad Torcuato Di Tella, Departamento de Matemática y Estadística, Av. Figueroa Alcorta 7350, C1428 Buenos Aires, Argentina}
\affil[2]{CONICET, Argentina}
\date{\small\texttt{dpinasco@utdt.edu}\qquad \texttt{victoria.venuti@utdt.edu}}

\begin{document}
\maketitle
\vspace{-1.0em}

\begin{abstract}
We establish a potential-theoretic sufficient condition for a compact subset of the real line to admit an optimal polynomial mesh. The criterion bounds the cardinality of a norming set by a linear term in the polynomial degree plus the number of connected components of a Green sublevel at height $\alpha/n$, where $\alpha>0$ is fixed. By combining this principle with an estimate due to Andrievskii, we prove that every uniformly perfect compact subset of $\R$ admits an optimal polynomial mesh. A standard product argument then produces optimal meshes on finite Cartesian
products; in particular, it yields an optimal mesh on the planar Cantor dust
$C\times C$, which is self-similar, totally disconnected, and has empty interior.
\end{abstract}



\section{Introduction}

A basic problem in polynomial approximation is to control the uniform norm of a polynomial on a compact set from its values on a finite subset. Let $\Pi_n$ denote the real algebraic polynomials of degree at most $n$. For a compact set $A\subset\R$ and a polynomial $p$, we write
\[
\norm{p}_A:=\sup_{x\in A}|p(x)|.
\]
If $E\subset\R$ is an infinite compact set, a finite subset $A_n\subset E$ is called a norming set for $\Pi_n$ with constant $\Lambda\ge1$ provided that
\[
\norm{p}_E\le\Lambda\norm{p}_{A_n},\qquad p\in\Pi_n.
\]
A sequence of such sets with the same constant $\Lambda$ independent of $n$ is an admissible polynomial mesh; it is optimal when $|A_n|=O(n)$, which is the natural order since $\dim\Pi_n=n+1$.

For the product construction in Section~\ref{sec:examples}, we denote by
$\mathcal P_n^d$ the real algebraic polynomials in $d$ variables of total
degree at most $n$. The relevant dimension for a compact set $K\subset\R^d$
is that of the restricted space $\mathcal P_n^d|_K$. If $K$ is
polynomial determining, then
\[
\dim(\mathcal P_n^d|_K)
=
\dim\mathcal P_n^d
=
\binom{n+d}{d}
\asymp n^d,
\]
so a polynomial mesh of cardinality $O(n^d)$ has the optimal possible order.

Polynomial meshes were introduced in the modern approximation-theoretic framework by Calvi and Levenberg \cite{CalviLevenberg2008}, where admissible and weakly admissible meshes arise naturally in the context of discrete least-squares approximation. Subsequent development by Bos, Calvi, Levenberg, Sommariva, and Vianello \cite{BosEtAl2011} emphasized geometric constructions and applications to discrete least-squares approximation, approximate Fekete points and interpolation. Dubiner's polynomial metric \cite{Dubiner1995}, which quantifies polynomial variation within a geometry adapted to the compact set, has also played an important role in constructions of optimal cardinality; see also the systematic metric framework developed by Bos, Levenberg, and Waldron \cite{BosLevenbergWaldron2008}.

The problem of achieving the smallest possible order of cardinality was highlighted by Kroó in \cite{Kroo2011}. Kroó established optimal meshes for several classes, including convex polytopes and sufficiently smooth star-like domains, and conjectured that every convex body admits an optimal polynomial mesh. Subsequent work proceeded along two complementary directions: reducing the boundary regularity required for star-like or star-shaped domains, and exploiting convexity to treat domains without smoothness assumptions on the boundary. Kroó derived Bernstein-type inequalities for star-like domains \cite{Kroo2013}; Piazzon addressed star-shaped Lipschitz domains under positive-reach assumptions and general $C^{1,1}$ domains \cite{Piazzon2016}; Kroó proved the convex-domain result in the plane using tangential Bernstein estimates \cite{Kroo2019}; and Dai and Prymak ultimately confirmed the conjecture in every dimension \cite{DaiPrymak2024}. The proof of Dai and Prymak relies on a boundary-adapted metric and separation-covering arguments rooted in Dubiner geometry.

Polynomial-mesh theory also extends well beyond smooth or convex domains, through Markov inequalities and weakly admissible meshes \cite{BosEtAl2011}. Stability under small perturbations on Markov compacts was investigated by Piazzon and Vianello \cite{PiazzonVianello2013}, while Tchakaloff meshes can be constructed from positive quadrature rules and Christoffel-function estimates \cite{BosVianello2019}. In contrast, the present paper concerns a different situation: compact subsets of the real line for which connected domain geometry may be entirely absent.

Let $E\subset\R$ be a regular compact set of positive logarithmic capacity, and let $g_E$ be the Green function of $\widehat\C\setminus E$ with pole at infinity. For $\alpha>0$ and $n\ge1$, define
\begin{equation}
K_{n,\alpha}:=\conv(E)\cap\{x\in\R:g_E(x)\le \alpha/n\},
\qquad
m_E(n,\alpha):=\#\comp K_{n,\alpha}.
\label{eq:intro-sublevel}
\end{equation}
Here $\conv(E)=[\min E,\max E]$ is the convex hull of $E$. Section~\ref{sec:prelim} recalls the potential-theoretic notation used in this definition.

For $\Lambda>1$, we define
\begin{equation}
M_E(n,\Lambda):=\min\Bigl\{|A|:A\subset E,\ 
\norm{p}_E\le \Lambda\norm{p}_A\ \text{for every }p\in\Pi_n\Bigr\}.
\label{eq:ME-def}
\end{equation}
For comparison, the classical Fekete-point power argument, with the
oversampling factor adjusted to the prescribed norming constant, shows
that for every infinite compact set $E\subset\R$ and every $\Lambda>1$,
\[
M_E(n,\Lambda)=O_\Lambda(n\log(n+1));
\]
see, for example, \cite{Bos2018}. Thus the linear estimates established in this paper remove the logarithmic oversampling present in this general construction.

Our main estimate is
\begin{equation}
M_E(n,\Lambda)
\le m_E(n,\alpha)
   +\frac{\pi e^\alpha}{1-\Lambda^{-1}}\,n.
\label{eq:main-intro}
\end{equation}
The first term measures the number of interval components of the Green sublevel and the second is the polynomial sampling cost. The proof distributes this linear cost among the components according to their equilibrium masses, whose sum is one.

Consequently, a linear bound for $m_E(n,\alpha)$ gives an optimal polynomial mesh. Andrievskii's work on Chebyshev polynomials provides precisely such an estimate for uniformly perfect compact subsets of the real line: after affine normalization, the Green sublevel at height $1/n$ has at most $c_E n$ interval components \cite[proof of Theorem~3, p.~521, immediately before (4.12)]{Andrievskii2017}. Combining this estimate with \eqref{eq:main-intro} yields optimal meshes
on every uniformly perfect compact subset of $\R$.

The class of uniformly perfect compact sets encompasses a wide range of geometries, including finite-gap sets and totally disconnected fractal sets. Examples of the latter include the classical ternary Cantor set, other fixed-ratio Cantor sets, the positive-measure Smith--Volterra--Cantor set, and real Cantor Julia sets arising from quadratic dynamics. In particular, for the ternary Cantor set $C$,
\[
n+1\le M_C(n,\Lambda)\le C_\Lambda n.
\]
Polynomial interpolation on Cantor-type sets has also been studied from the different viewpoint of Lebesgue constants: Goncharov and Paksoy prove unboundedness for three natural Cantor families \cite{GoncharovPaksoy2025}. Their result concerns interpolation stability rather than the degree-uniform norming problem considered here.

Cartesian products yield immediate higher-dimensional examples. If
$E_1,\ldots,E_d\subset\R$ admit optimal polynomial meshes, the standard
product construction gives an optimal mesh of cardinality $O(n^d)$ on
$E_1\times\cdots\times E_d$. In particular, the ternary Cantor result
yields optimal meshes on $C^d$; for $d=2$ this includes the planar Cantor
dust $C\times C$.

The proof of \eqref{eq:main-intro} follows the univariate principle emphasized by Vianello: a Bernstein-type derivative inequality can be integrated in its natural cumulative coordinate and then discretized there \cite[Proposition~1 and Remark~1, pp.~930--931]{Vianello2014}. Here the auxiliary compact set depends on $n$ and may have many connected components, while the final sampling points are required to belong to the original compact set $E$. Cumulative equilibrium mass provides the bookkeeping needed to satisfy both requirements. The resulting nodes also form an $O(1/n)$-net for the intrinsic Dubiner metric of $E$.

On a single interval the construction reduces to the classical Bernstein framework: the equilibrium measure is the arcsine measure and uniform spacing in cumulative equilibrium mass gives the Chebyshev--Lobatto points. This interval model also identifies the equilibrium coordinate exactly with the Dubiner distance, up to normalization.

Brudnyi and Yomdin investigated norming sets and Remez-type inequalities
for fixed finite-dimensional spaces of continuous functions
\cite{BrudnyiYomdin2016}. In particular, if $V$ has dimension $\ell$ and
$Z$ is $V$-norming, their finite extraction result yields a subset
$Z_0\subset Z$ of cardinality $\ell$, with a possible loss by a factor
$\ell$ in the norming constant. Thus, when $V=\Pi_n$, this argument does
not provide a degree-uniform admissible mesh, since $\ell=n+1$. More recently, Bia\l as-Cie\.z, Kowalska and Sommariva list uniformly
perfect sets and certain Cantor sets among examples satisfying a division
inequality in their study of polynomial meshes on algebraic sets
\cite{BialasCiezEtAl2026}. Their lifting constructions transfer meshes
from a projected compact base to subsets of algebraic varieties; in
particular, they require a mesh on the base and therefore do not by
themselves provide an optimal polynomial mesh on a Cantor base. The Green-level transfer estimate obtained here turns potential-theoretic control of the sublevels into a degree-uniform norming construction on the original compact set. Combined with Andrievskii's component estimate, it yields optimal polynomial meshes for every uniformly perfect compact subset of the real line.

The remainder of the paper is organized as follows. Section~\ref{sec:prelim} collects the potential-theoretic notation and the two polynomial inequalities used in the main argument. Section~\ref{sec:transfer} analyzes the Green sublevels and proves \eqref{eq:main-intro}, together with its interval and Dubiner interpretations. Section~\ref{sec:up} applies Andrievskii's estimate to uniformly perfect compact sets. Section~\ref{sec:examples} discusses representative irregular examples and finite products.

\section{Preliminaries}\label{sec:prelim}

\subsection{Capacity, equilibrium measure and Green function}

Let $E\subset\C$ be a compact set and let $\mathcal M_1(E)$ denote the Borel probability measures supported on $E$. For $\mu\in\mathcal M_1(E)$ its logarithmic potential and energy are
\[
U^\mu(z):=\int_E\log\frac1{|z-t|}\,d\mu(t),
\qquad
I(\mu):=\iint_{E\times E}\log\frac1{|x-y|}\,d\mu(x)d\mu(y).
\]
The Robin constant and logarithmic capacity are defined by
\begin{equation}
V_E:=\inf_{\mu\in\mathcal M_1(E)}I(\mu),
\qquad
\caplog(E):=e^{-V_E}.
\label{eq:capacity}
\end{equation}
We say that $E$ has positive logarithmic capacity when $\caplog(E)>0$, equivalently $V_E<\infty$. In this case, there exists a unique equilibrium measure $\mu_E\in\mathcal M_1(E)$ with $I(\mu_E)=V_E$; see, for example, Saff \cite[Lemma~1.7 and Definition~1.8, pp.~172--173]{Saff2010}.

For a compact set $E\subset\R$ of positive logarithmic capacity,
$\Omega=\widehat\C\setminus E$ is connected. Its Green function with
pole at infinity is denoted by
\[
g_E(z)=G_\Omega(z,\infty).
\]
With the standard normalization, one has
\begin{equation}
g_E(z)=V_E-U^{\mu_E}(z)
\quad (z\in\Omega),
\qquad
g_E(z)=\log|z|-\log\caplog(E)+o(1)
\quad(z\to\infty),
\label{eq:green-potential}
\end{equation}
see \cite[Definition~3.4 and (3.1), p.~184]{Saff2010}. The set $E$ is called \emph{regular} when $g_E$ extends continuously to $E$ with boundary value zero. For regular $E$, we use the same notation $g_E$ for its continuous
extension by zero to $E$.

\subsection{Bernstein--Walsh and Totik's inequality}

Our first standard tool is the Bernstein--Walsh inequality. If $E$ is regular and of positive logarithmic capacity, then for every $p\in\Pi_n$,
\begin{equation}
|p(z)|\le \norm{p}_E e^{n g_E(z)},
\qquad z\in\C,
\label{eq:BW}
\end{equation}
where on $E$ the inequality is understood using the continuous extension $g_E=0$; see Saff \cite[Lemma~3.7, p.~186]{Saff2010}.

The second input concerns a finite union of nondegenerate closed intervals
\[
K=\bigcup_{j=1}^{\ell}[a_j,b_j],
\qquad
a_1<b_1<a_2<\cdots<a_\ell<b_\ell.
\]
For such a set, the equilibrium measure is absolutely continuous with
respect to Lebesgue measure. We write
\[
d\mu_K(x)=\omega_K(x)\,dx.
\]
Totik gives an explicit formula for the density $\omega_K$ in
\cite[(2.4), p.~143, and (2.8), p.~144]{Totik2001}. More importantly for
our purposes, his exact Bernstein--Szeg\H{o} inequality implies that, for
every polynomial $p$ of degree $m\ge1$,
\begin{equation}
\left(\frac{p'(x)}{\pi\omega_K(x)}\right)^2
+m^2p(x)^2
\le m^2\norm{p}_K^2,
\qquad x\in\operatorname{Int}K,
\label{eq:Totik-strong}
\end{equation}
and hence
\begin{equation}
|p'(x)|\le m\pi\omega_K(x)\norm{p}_K,
\qquad x\in\operatorname{Int}K.
\label{eq:Totik}
\end{equation}
See \cite[Theorem~3.1, (3.3)--(3.4), p.~149]{Totik2001}. In particular,
the coefficient in \eqref{eq:Totik} is independent of the number,
lengths, and relative positions of the interval components.

\section{Green sublevels and the transfer estimate}\label{sec:transfer}

Let $E\subset\R$ be a regular compact set of positive logarithmic capacity, and denote $I=\conv(E)$. For $t>0$ define
\begin{equation}
K_E(t):=I\cap\{x\in\R:g_E(x)\le t\}.
\label{eq:Kt}
\end{equation}
The elementary structure of these sublevels is the reason that the finite-union inequality \eqref{eq:Totik} can be used.

\begin{lemma}\label{lem:gaps}
For every $t>0$, $K_E(t)$ is a finite union of nondegenerate closed intervals, each of which intersects $E$. Moreover, if $H=(a,b)$ is a gap of $E$ within $I$, then $g_E$ is strictly concave on $H$, and $\{x\in H:g_E(x)>t\}$ is either empty or a single open interval.
\end{lemma}

\begin{proof}
Fix a compact subinterval $H'\subset H=(a,b)$. Since $\dist(H',E)>0$, the first two derivatives of $\log|x-s|$ are uniformly dominated on $H'\times E$, so differentiation under the integral sign is justified. As $H'$ is arbitrary, \eqref{eq:green-potential} gives
\[
(U^{\mu_E})''(x)=\int_E\frac{d\mu_E(s)}{(x-s)^2}>0,
\qquad x\in H,
\]
and hence
\begin{equation}
g_E''(x)=-\int_E\frac{d\mu_E(s)}{(x-s)^2}<0,
\qquad x\in H.
\label{eq:concavity}
\end{equation}
Thus $g_E$ is strictly concave on $H$. Since $E$ is regular, $g_E(a)=g_E(b)=0$, and every positive superlevel set in $H$ is either empty or a single open interval. It follows that $K_E(t)$ is obtained from $I$ by deleting pairwise disjoint open intervals, one from each gap whose Green height exceeds $t$.

Only finitely many such gaps occur for a fixed $t>0$. Since $g_E$ is continuous on the compact interval $I$ and vanishes on $E$, there exists $\eta>0$ such that $\dist(x,E)<\eta$ implies $g_E(x)<t$. If a gap $H=(a,b)$ has Green height greater than $t$, it contains a point $x$ with $g_E(x)>t$, hence $\dist(x,E)\ge\eta$. Since $a,b\in E$, this gives $b-a\ge2\eta$. The disjoint gaps lie in the bounded interval $I$, so there can be only finitely many gaps of length at least $2\eta$.

Finally, every point of $E$ has a neighbourhood in $I$ on which $g_E<t$. Hence a component of $K_E(t)$ meeting $E$ is nondegenerate. Conversely, the strict concavity just proved prevents a component of the sublevel set from lying strictly inside a gap without reaching one of its endpoints. Thus every component meets $E$.
\end{proof}

For fixed $\alpha>0$ and $n\ge1$ we write
\begin{equation}
K_{n,\alpha}=K_E(\alpha/n),
\qquad
m_E(n,\alpha)=\#\comp K_{n,\alpha}.
\label{eq:m-def}
\end{equation}
By Lemma~\ref{lem:gaps}, this number is finite. Bernstein--Walsh gives, for $p\in\Pi_n$,
\begin{equation}
\norm{p}_{K_{n,\alpha}}\le e^\alpha\norm{p}_E.
\label{eq:BW-sublevel}
\end{equation}

\begin{theorem}\label{thm:transfer}
Let $E\subset\R$ be a regular compact set of positive logarithmic capacity. Fix $\alpha>0$ and $\Lambda>1$. Then for every $n\ge1$, there exists a finite set $A_n\subset E$ such that
\begin{equation}
\norm{p}_E\le \Lambda\norm{p}_{A_n},
\qquad p\in\Pi_n,
\label{eq:norming-main}
\end{equation}
and
\begin{equation}
|A_n|
\le
m_E(n,\alpha)
+\frac{\pi e^\alpha}{1-\Lambda^{-1}}\,n.
\label{eq:card-main}
\end{equation}
Consequently, if $m_E(n,\alpha)=O(n)$ for one fixed $\alpha>0$, then $E$ admits an optimal polynomial mesh.
\end{theorem}

\begin{proof}
Write $K_{n,\alpha}=J_1\cup\cdots\cup J_m$, where $m=m_E(n,\alpha)$ and the $J_j$ are the pairwise disjoint interval components given by Lemma~\ref{lem:gaps}. Let $\nu_n$ be the equilibrium measure of $K_{n,\alpha}$ and let $\omega_n$ denote its density.

For each $J_j$ define
\begin{equation}
F_j(x):=\nu_n\bigl(J_j\cap(-\infty,x]\bigr),
\qquad x\in J_j.
\label{eq:Fj}
\end{equation}
Since $\nu_n$ is absolutely continuous, $F_j$ is continuous on $J_j$.
Hence
\[
Y_j:=F_j(E\cap J_j)
\]
is compact. We discretize only $Y_j$, so that every selected node will belong to $E$. Put
\begin{equation}
\beta:=\frac{1-\Lambda^{-1}}{\pi e^\alpha}.
\label{eq:beta}
\end{equation}
Choose a maximal $\beta/n$-separated subset of $Y_j$. It is automatically a $\beta/n$-net of $Y_j$. Since $Y_j$ is contained in an interval of length $\nu_n(J_j)$, the number of selected values is at most $1+n\nu_n(J_j)/\beta$. For each selected value choose one preimage in $E\cap J_j$, and call the resulting set $A_{n,j}$.

With $A_n=\bigcup_{j=1}^m A_{n,j}$, the total equilibrium mass gives
\begin{equation}
|A_n|
\le
m+\frac n\beta\sum_{j=1}^m\nu_n(J_j)
=
m+\frac n\beta,
\label{eq:card-sum}
\end{equation}
which is \eqref{eq:card-main}.

Fix $x\in E\cap J_j$. There exists $a\in A_{n,j}$ such that
$|F_j(x)-F_j(a)|\le\beta/n$. Since $\nu_n$ is absolutely continuous, it has no atoms, and therefore
\begin{equation}
\nu_n([x,a])=|F_j(x)-F_j(a)|\le\frac\beta n,
\label{eq:mass-close}
\end{equation}
where $[x,a]$ denotes the interval between the two points. Integrating \eqref{eq:Totik} and then using \eqref{eq:BW-sublevel},
\begin{equation}
\begin{aligned}
|p(x)-p(a)|
&\le n\pi\norm{p}_{K_{n,\alpha}}\nu_n([x,a])\\
&\le \pi\beta e^\alpha\norm{p}_E
=(1-\Lambda^{-1})\norm{p}_E.
\end{aligned}
\label{eq:oscillation}
\end{equation}
This argument remains valid when $x$ or $a$ is an endpoint of a component: \eqref{eq:Totik} holds almost everywhere in its interior, while
$\omega_n$ is integrable and the endpoints have Lebesgue measure zero.

Choose $x_0\in E$ with $|p(x_0)|=\norm{p}_E$ and take the corresponding $a\in A_n$. Then \eqref{eq:oscillation} gives $|p(a)|\ge\Lambda^{-1}\norm{p}_E$, proving \eqref{eq:norming-main}.
\end{proof}

The construction is a degree-dependent version of the one-dimensional discretization principle in Vianello \cite[Proposition~1 and Remark~1, pp.~930--931]{Vianello2014}. The point specific to the present setting is that $K_{n,\alpha}$ may have many components and the nodes are nevertheless selected on $E$. The mass identity $\sum_j\nu_n(J_j)=1$ turns the componentwise sampling into a single linear budget.

\begin{remark}\label{rem:alpha}
The qualitative condition in the last sentence of Theorem~\ref{thm:transfer} is independent of the fixed positive level parameter. If $m_E(n,\alpha_0)=O(n)$ for some $\alpha_0>0$, then $m_E(n,\alpha)=O(n)$ for every $\alpha>0$. Indeed, by Lemma~\ref{lem:gaps}, the number of components of $K_E(t)$ is nonincreasing in $t$. If $\alpha\ge\alpha_0$ the conclusion is immediate. If $0<\alpha<\alpha_0$, let $N=\lceil\alpha_0n/\alpha\rceil$. Since $\alpha_0/N\le\alpha/n$,
\[
m_E(n,\alpha)\le m_E(N,\alpha_0)=O(N)=O(n).
\]
\end{remark}

\subsection{Equilibrium mass and intrinsic Dubiner geometry}\label{subsec:dubiner}

The preceding sampling argument was designed to compare polynomial values directly. It also has a geometric interpretation. The Dubiner metric measures how much angular variation a polynomial of bounded norm can create per unit of degree, and nets at scale $1/n$ for this metric are a standard source of norming sets. In the present construction the natural angular coordinate first appears on the Green filling $K_{n,\alpha}$ through its equilibrium measure. The point of this subsection is to show that the same nodes are, in fact, a net at scale $1/n$ for the intrinsic Dubiner metric of the original compact set $E$.

For a compact real set $K$, we use the normalization
\begin{equation}
d_K^D(x,y):=
\sup_{\substack{p\in\R[x],\ \deg p\ge1\\ \norm{p}_K\le1}}
\frac{|\arccos p(x)-\arccos p(y)|}{\deg p},
\qquad x,y\in K.
\label{eq:Dubiner}
\end{equation}
This normalization goes back to \cite{Dubiner1995}; see also
\cite{BosLevenbergWaldron2008} for the associated metric framework and
\cite{Vianello2018} in the norming-mesh setting. Although the Dubiner
distance is naturally defined as a pseudometric in general, linear
polynomials separate points of a compact subset of $\R$, so $d_K^D$ is
a genuine metric in the present setting. Its relevance for meshes is immediate: if $A_n\subset K$ is a $\theta/n$-net for $d_K^D$, with $0<\theta<\pi/2$, then
\[
\norm{p}_K\le \frac{1}{\cos\theta}\norm{p}_{A_n},
\qquad p\in\Pi_n.
\]
Indeed, after normalizing $\norm{p}_K=1$ and choosing a point where $|p|=1$, a node at Dubiner distance at most $\theta/n$ changes the angle $\arccos p$ by at most $\theta$. Thus small Dubiner nets are precisely adapted to the norming problem.

For a finite union of intervals, Totik's strong Bernstein inequality makes the relation with equilibrium mass explicit. Let $K$ have equilibrium density $\omega_K$, and let $x,y$ lie in the same interval component. Then
\begin{equation}
d_K^D(x,y)
\le \pi\left|\int_x^y\omega_K(t)\,dt\right|
=\pi\mu_K([x,y]).
\label{eq:dubiner-mass}
\end{equation}
To see this, let $p$ have degree $m\ge1$ and $\norm{p}_K\le1$. From \eqref{eq:Totik-strong},
\[
|p'(t)|\le m\pi\omega_K(t)\sqrt{1-p(t)^2}
\]
almost everywhere in the interior. Applying the chain rule to $\arccos(\rho p)$, $0<\rho<1$, avoids the endpoint singularities of $\arccos$ and gives
\[
\left|\frac{d}{dt}\arccos(\rho p(t))\right|
\le m\pi\omega_K(t)
\frac{\rho\sqrt{1-p(t)^2}}{\sqrt{1-\rho^2p(t)^2}}
\le m\pi\omega_K(t).
\]
Integration from $x$ to $y$, followed by $\rho\uparrow1$, division by $m$ and the supremum over $p$, proves \eqref{eq:dubiner-mass}.

For the nodes constructed in Theorem~\ref{thm:transfer}, \eqref{eq:mass-close} therefore yields, for every $x\in E$, a point $a\in A_n$ in the same component of $K_{n,\alpha}$ such that
\[
d_{K_{n,\alpha}}^D(x,a)
\le \frac{\pi\beta}{n}
=\frac{1-\Lambda^{-1}}{e^\alpha n}.
\]
At this stage the metric is still that of the Green filling. Since $E\subset K_{n,\alpha}$, monotonicity goes in the opposite direction from the one needed here:
\[
d_E^D(x,y)\ge d_{K_{n,\alpha}}^D(x,y),
\qquad x,y\in E.
\]
Thus the preceding estimate does not by itself control the intrinsic geometry of $E$.

Bernstein--Walsh provides the missing information: a polynomial normalized on $E$ grows only exponentially with its degree on a Green sublevel. A direct normalization, however, interacts poorly with $\arccos$ near $\pm1$. The remedy is to amplify angular separation first. If
\[
p(x)=\cos\theta,\qquad p(y)=\cos\phi,
\]
then the Chebyshev polynomial $T_r$ satisfies
\[
T_r(p(x))=\cos(r\theta),\qquad T_r(p(y))=\cos(r\phi).
\]
The next lemma says that a suitable frequency $r$ of order $1/|\theta-\phi|$ turns any nonzero angular separation into a fixed separation of polynomial values. The numerical constants are deliberately nonessential.

\begin{lemma}\label{lem:angular-amplification}
For every $\theta,\phi\in[0,\pi]$ with $\delta=|\theta-\phi|>0$, there is an integer $r$ such that
\[
1\le r\le\frac{16}{\delta},
\qquad
|\cos(r\theta)-\cos(r\phi)|\ge\frac12.
\]
\end{lemma}

\begin{proof}
Put
\[
A=\frac{\theta+\phi}{2},
\qquad
b=\frac{|\theta-\phi|}{2}.
\]
Using
\[
\cos u-\cos v
=
-2\sin\left(\frac{u+v}{2}\right)
 \sin\left(\frac{u-v}{2}\right),
\]
we obtain
\[
|\cos(r\theta)-\cos(r\phi)|
=
2|\sin(rA)\sin(rb)|.
\]
The issue is that choosing $r$ of order $1/b$ makes the second sine
nontrivial but does not prevent a simultaneous resonance in the first
one. We average over a short block of frequencies to rule this out.

Let $N=\lceil2\pi/b\rceil$, $w_r=1-r/(N+1)$, and
\[
D_N(t)=\sum_{r=1}^Nw_r\cos(rt).
\]
The classical Fej\'er identity
\[
1+2D_N(t)
=\frac{1}{N+1}
 \left(\frac{\sin((N+1)t/2)}{\sin(t/2)}\right)^2
\ge0
\]
shows in particular that $D_N(t)\ge-1/2$. It also gives the required upper bounds at the two relevant frequencies. Indeed,
\[
D_N(t)\le
\frac{1}{2(N+1)\sin^2(t/2)}.
\]
Since $0<b\le\pi/2$, $A\in[b,\pi-b]$, $\sin A\ge\sin b\ge2b/\pi$, and $Nb\ge2\pi$, it follows that
\[
D_N(2A)\le\frac{N}{32},
\qquad
D_N(2b)\le\frac{N}{32}.
\]
Now use
\[
\begin{aligned}
\sin^2(rA)\sin^2(rb)
={}&\frac14-\frac14\cos(2rA)-\frac14\cos(2rb)\\
&+\frac18\cos(2r(A-b))+\frac18\cos(2r(A+b)).
\end{aligned}
\]
After multiplying by $w_r$ and summing, the Fej\'er bounds imply
\[
\sum_{r=1}^Nw_r\sin^2(rA)\sin^2(rb)
\ge \frac{N}{8}-\frac{N}{64}-\frac18
=\frac{7N}{64}-\frac18
\ge\frac{5N}{64},
\qquad N\ge4.
\]
Since $\sum_{r=1}^Nw_r=N/2$, at least one $r\le N$ satisfies
\[
\sin^2(rA)\sin^2(rb)\ge\frac5{32}.
\]
For this $r$,
\[
|\cos(r\theta)-\cos(r\phi)|
\ge\sqrt{\frac58}>\frac12.
\]
Finally, since $b=\delta/2$ and $\delta\le\pi$,
\[
r\le N
\le \frac{2\pi}{b}+1
= \frac{4\pi}{\delta}+1
\le \frac{5\pi}{\delta}
<\frac{16}{\delta}.
\]
\end{proof}

The amplification lemma leads to a comparison principle in which the
intrinsic Dubiner distance on $E$ is controlled by the larger of the
Bernstein--Walsh thickening parameter and the Dubiner distance on $K$.
The proof separates the two cases according to which of these quantities
controls the amplified polynomial.

\begin{proposition}\label{prop:dubiner-thickening}
Let $E\subset K\subset\R$ be compact and suppose that, for some $\tau\ge0$,
\[
\norm{q}_K\le e^{\tau\deg q}\norm{q}_E
\]
for every nonconstant real polynomial $q$. Then there is an absolute constant $C$ such that
\[
d_E^D(x,y)
\le C\max\{\tau,d_K^D(x,y)\},
\qquad x,y\in E.
\]
For instance, one may take $C=64$.
\end{proposition}

\begin{proof}
Let $p$ have degree $m\ge1$ and $\norm{p}_E\le1$, and put
\[
\theta=\arccos p(x),\qquad
\phi=\arccos p(y),\qquad
\delta=|\theta-\phi|.
\]
There is nothing to prove if $\delta=0$. Otherwise choose $r$ from Lemma~\ref{lem:angular-amplification} and set
\[
P=T_r\circ p,\qquad s=\tau rm.
\]
Then
\[
\deg P=rm,\qquad \norm{P}_E\le1,
\qquad |P(x)-P(y)|\ge\frac12,
\qquad \norm{P}_K\le e^s.
\]

If $s\ge1/2$, then $r\le16/\delta$ gives
\[
\frac12\le \tau rm\le\frac{16\tau m}{\delta},
\]
and hence
\[
\frac{\delta}{m}\le32\tau.
\]

Suppose now that $s<1/2$ and normalize by setting $Q=e^{-s}P$. Then $\norm{Q}_K\le1$, while the fixed separation survives:
\[
|Q(x)-Q(y)|\ge\frac12e^{-1/2}.
\]
Since cosine is $1$-Lipschitz on $[0,\pi]$,
\[
|\arccos u-\arccos v|\ge|u-v|,
\qquad u,v\in[-1,1].
\]
Therefore the definition of $d_K^D$ gives
\[
rm\,d_K^D(x,y)
\ge |\arccos Q(x)-\arccos Q(y)|
\ge\frac12e^{-1/2}.
\]
Using again $r\le16/\delta$,
\[
\frac{\delta}{m}
\le32e^{1/2}d_K^D(x,y)
<64d_K^D(x,y).
\]
The two cases give
\[
\frac{|\arccos p(x)-\arccos p(y)|}{m}
\le64\max\{\tau,d_K^D(x,y)\}.
\]
Taking the supremum over $p$ proves the result.
\end{proof}

For the Green filling, Bernstein--Walsh gives
\[
\norm{q}_{K_{n,\alpha}}
\le e^{(\alpha/n)\deg q}\norm{q}_E.
\]
Thus Proposition~\ref{prop:dubiner-thickening} applies with $\tau=\alpha/n$. Combining it with the equilibrium-mass estimate above, we obtain, for every $x\in E$, a point $a\in A_n$ such that
\begin{equation}
d_E^D(x,a)
\le\frac{C}{n}
\max\left\{\alpha,(1-\Lambda^{-1})e^{-\alpha}\right\},
\label{eq:intrinsic-dubiner-net}
\end{equation}
where $C$ is an absolute constant. Hence the same nodes form an $O(1/n)$-net for the intrinsic Dubiner metric of $E$.

This identifies the geometric content of the construction. The Green filling makes the disconnected set accessible to the exact Bernstein inequality; equilibrium mass supplies a global sampling coordinate on the filling; and the comparison above transfers the resulting polynomial-scale covering back to the intrinsic geometry of $E$. The direct proof of Theorem~\ref{thm:transfer} remains quantitatively sharper, while \eqref{eq:intrinsic-dubiner-net} shows that its nodes also recover the standard Dubiner covering mechanism behind norming meshes.

\subsection{Recovery of the classical interval construction}\label{subsec:interval}

Take $E=[-1,1]$. Since $g_E=0$ on $E$,
\[
K_{n,\alpha}=E,\qquad m_E(n,\alpha)=1.
\]
The equilibrium measure of $[-1,1]$ is the classical arcsine measure
\[
d\mu_E(x)=\frac{dx}{\pi\sqrt{1-x^2}},
\qquad -1<x<1;
\]
see, for example, \cite{Saff2010}.
Substituting the arcsine density into \eqref{eq:Totik} recovers the
classical Bernstein inequality, for every $p\in\Pi_n$,
\begin{equation}
|p'(x)|\le \frac{n}{\sqrt{1-x^2}}\norm{p}_{[-1,1]},
\qquad -1<x<1.
\label{eq:classical-Bernstein}
\end{equation}
The cumulative equilibrium coordinate is
\begin{equation}
F(x)=\mu_E([-1,x])
=\frac1\pi\left(\arcsin x+\frac\pi2\right),
\qquad -1\le x\le1.
\label{eq:arcsine-coordinate}
\end{equation}
Hence, for any integer $m\ge1$, choosing $F(x_j)=j/m$ gives
\begin{equation}
x_j=-\cos\frac{j\pi}{m},\qquad j=0,\ldots,m,\label{eq:CL-points}
\end{equation}
the Chebyshev--Lobatto points; compare Vianello \cite[pp.~931--932]{Vianello2014}.

In this case the Dubiner interpretation is exact. From \eqref{eq:dubiner-mass} and the test polynomial $p(t)=t$,
\begin{equation}
d_{[-1,1]}^D(x,y)
=|\arccos x-\arccos y|
=\pi|F(x)-F(y)|.
\label{eq:interval-dubiner}
\end{equation}
Thus the equilibrium coordinate is precisely the angular coordinate underlying both the Dubiner geometry and Chebyshev--Lobatto sampling. In particular, the general $O(1/n)$-net conclusion above becomes here exactly the familiar angular spacing of the Chebyshev--Lobatto points.

\section{Uniformly perfect compact sets}\label{sec:up}

A compact set $E\subset\R$ containing at least two points is said to be
\emph{uniformly perfect} if there exists $\gamma\in(0,1)$ such that for
every $x\in E$ and every $0<r<\diam(E)$, the annulus
\[
\{y\in\R:\gamma r\le |x-y|\le r\}
\]
contains a point of $E$. Thus, at every point and every scale, $E$
contains another point at a distance comparable with that scale.
In particular, a uniformly perfect compact set has no isolated points
and is therefore infinite. Geometrically, uniform perfectness rules out
arbitrarily large relative gaps across successive scales, without
requiring connectedness, positive Lebesgue measure, or nonempty interior.

Andrievskii recalls an equivalent formulation in terms of logarithmic
capacity: there exists $\lambda_E>0$ such that
\begin{equation}
\caplog\bigl(E\cap\{y\in\R:|y-x|\le r\}\bigr)
\ge \lambda_E r,
\qquad
x\in E,\quad 0<r<\diam(E);
\label{eq:UP-capacity}
\end{equation}
see \cite[(1.5)]{Andrievskii2017}. In particular,
$\caplog(E)>0$. He also recalls that this capacity-density condition
implies regularity of the Green function: $g_E$ extends continuously
to $E$ with boundary value zero; see the discussion preceding (4.2)
in \cite[p.~517]{Andrievskii2017}. Thus uniformly perfect compact sets
satisfy the potential-theoretic hypotheses of
Theorem~\ref{thm:transfer}.

The remaining ingredient is the number of components of the relevant
Green sublevels. After affine normalization $I=[-1,1]$, Andrievskii
considers
\[
K_n^*=I\cap\{g_E\le 1/n\}.
\]
In the proof of \cite[Theorem~3, p.~521, immediately before
(4.12)]{Andrievskii2017}, for a compact set with infinitely many
components, it is shown that $K_n^*$ is the union of $N(E,n)+1$
disjoint closed intervals and that
\begin{equation}
N(E,n)+1\le c_E n.
\label{eq:Andrievskii}
\end{equation}

\begin{corollary}\label{cor:UP}
Let $E\subset\R$ be a uniformly perfect compact set. Then, for every
$\Lambda>1$, there exist a constant $C_{E,\Lambda}$ and finite sets
$A_n\subset E$ such that
\begin{equation}
|A_n|\le C_{E,\Lambda}n,
\qquad
\norm{p}_E\le\Lambda\norm{p}_{A_n}
\quad(p\in\Pi_n).
\label{eq:UPmesh}
\end{equation}
In particular, every uniformly perfect compact subset of the real line
admits an optimal polynomial mesh.
\end{corollary}

\begin{proof}
Suppose first that $E$ has infinitely many connected components.
Then \eqref{eq:Andrievskii} gives
$m_E(n,1)=O_E(n)$ after affine normalization, and
Theorem~\ref{thm:transfer} applies.

Suppose instead that $E$ has finitely many connected components.
Each connected component of a compact subset of $\R$ is a closed
interval, possibly degenerate. Since there are only finitely many
components, a singleton component would be isolated in $E$, which is
impossible for a uniformly perfect set. Hence $E$ is a finite union
of nondegenerate closed intervals.

By Lemma~\ref{lem:gaps}, every component of $K_{n,1}$ meets $E$.
Moreover, each connected component of $E$, being connected and
contained in $K_{n,1}$, lies in a single component of $K_{n,1}$.
Consequently, distinct components of $K_{n,1}$ meet distinct
components of $E$, and $m_E(n,1)$ is bounded by the number of
connected components of $E$. Theorem~\ref{thm:transfer} therefore
applies in this case as well.
\end{proof}

\section{Examples and finite products}\label{sec:examples}

Let $C\subset[0,1]$ be the classical middle-third Cantor set. Its uniform perfectness follows directly from its construction. Given $x\in C$ and $0<r<1$, choose $k$ so that $3^{-k}\le r<3^{-k+1}$ and put $L=3^{-k}$. The level-$k$ basic interval containing $x$ has length $L$. Choose the endpoint $y\in C$ that is farther from $x$. Then $|x-y|\ge L/2>r/6$, while $|x-y|\le L\le r$. Thus
\begin{equation}
\frac r6<|x-y|\le r.
\label{eq:C-UP}
\end{equation}

\begin{corollary}\label{cor:Cantor}
For the ternary Cantor set $C$ and every $\Lambda>1$,
\begin{equation}
n+1\le M_C(n,\Lambda)\le C_\Lambda n.
\label{eq:Cantor-main}
\end{equation}
In particular, the Cantor set $C$ admits an optimal polynomial mesh.
\end{corollary}

\begin{proof}
The upper bound follows from Corollary~\ref{cor:UP}. The lower bound is the univariate dimension obstruction: if $A\subset C$ has at most $n$ points, a nonzero polynomial of degree at most $n$ can vanish on $A$ but not on the infinite set $C$.
\end{proof}

The same conclusion holds for several other Cantor-type compact sets. Standard fixed-ratio Cantor sets are uniformly perfect by a direct comparison of successive scales. The Smith--Volterra--Cantor set \cite[p.~202]{Simon2015}, which has positive Lebesgue measure and empty interior, is also uniformly perfect: if $L_n=(2^n+1)/2^{2n+1}$ is the common length of its level-$n$ basic intervals, then $L_n/L_{n-1}\ge3/8$, and the farther-endpoint argument gives $3r/16<|x-y|\le r$. Thus the Smith–Volterra–Cantor set also admits an optimal polynomial mesh.

A different source of examples comes from complex dynamics. Given a polynomial $P$ of degree at least two, its filled Julia set is
\[
K(P):=\{z\in\C:(P^{\circ n}(z))_{n\ge0}\ \text{remains bounded}\},
\]
and its Julia set is $J(P):=\partial K(P)$. For the real quadratic family $f_c(z)=z^2+c$ with $c<-2$, $J(f_c)$ is a Cantor subset of the real line \cite{Lyubich2006}. Julia sets of polynomials of degree at least two are uniformly perfect by Hinkkanen's theorem \cite{Hinkkanen1994}. Therefore, Corollary~\ref{cor:UP} gives optimal polynomial meshes on these dynamically generated real Cantor sets as well.

For completeness, we record the standard product principle.

\begin{proposition}\label{prop:product}
Let $E_1,\ldots,E_d\subset\R$ be infinite compact sets. Suppose that for each $j$ there are finite sets $A_{j,n}\subset E_j$, constants $\Lambda_j\ge1$ independent of $n$, and constants $C_j$ such that
\[
\norm{q}_{E_j}\le\Lambda_j\norm{q}_{A_{j,n}}
\quad(q\in\Pi_n),
\qquad
|A_{j,n}|\le C_j n.
\]
Then
\[
A_n:=A_{1,n}\times\cdots\times A_{d,n}
\]
is an optimal polynomial mesh on $E_1\times\cdots\times E_d$ for total-degree polynomials, with norming constant $\prod_{j=1}^d\Lambda_j$ and cardinality $O(n^d)$.
\end{proposition}

\begin{proof}
For $d=2$ and $p\in\mathcal P_n^2$,
\[
\begin{aligned}
\norm{p}_{E_1\times E_2}
&=\sup_{y\in E_2}\sup_{x\in E_1}|p(x,y)|\\
&\le\Lambda_1\sup_{y\in E_2}\sup_{a\in A_{1,n}}|p(a,y)|\\
&\le\Lambda_1\Lambda_2
\sup_{b\in A_{2,n}}\sup_{a\in A_{1,n}}|p(a,b)|\\
&=\Lambda_1\Lambda_2\norm{p}_{A_{1,n}\times A_{2,n}}.
\end{aligned}
\]
At each step, the remaining one-variable polynomial has degree at most $n$. Repeating the same argument in the remaining variables gives
\[
\norm{p}_{E_1\times\cdots\times E_d}
\le\Bigl(\prod_{j=1}^d\Lambda_j\Bigr)
\norm{p}_{A_n}.
\]
The argument in fact works for polynomials of degree at most $n$ in each coordinate, and hence in particular for total degree at most $n$. Also
\[
|A_n|=\prod_{j=1}^d|A_{j,n}|=O(n^d).
\]
Finally, $E_1\times\cdots\times E_d$ is polynomial determining: successive fixation of all but one variable shows that a polynomial vanishing on the product is identically zero. Hence
\[
\dim(\mathcal P_n^d|_{E_1\times\cdots\times E_d})=\binom{n+d}{d}\asymp n^d,
\]
so the cardinality order is optimal.
\end{proof}

This product stability is standard in polynomial-mesh theory; see \cite{BosEtAl2011}.

\begin{corollary}\label{cor:Cantor-product}
For every $d\ge1$ and every $\Lambda>1$, there are finite sets $A_n\subset C^d$ and a constant $C_{\Lambda,d}$ such that
\[
|A_n|\le C_{\Lambda,d}n^d,
\qquad
\norm{p}_{C^d}\le\Lambda\norm{p}_{A_n}
\quad(p\in\mathcal P_n^d).
\]
In particular, the product Cantor set $C^d$ admits an optimal polynomial mesh.
\end{corollary}

\begin{proof}
Apply Proposition~\ref{prop:product} using on each factor a mesh from Corollary~\ref{cor:Cantor} with norming constant $\Lambda^{1/d}$.
\end{proof}

For $d=2$ this gives the planar Cantor dust $C\times C$. It is totally disconnected, has empty interior, and Hausdorff dimension $2\log2/\log3$, while its boundary is the set itself. Thus the one-dimensional potential-theoretic result yields optimal meshes on compact sets with geometry very different from the full-dimensional domains appearing in the classical geometric constructions.

The examples in this section range from zero-measure self-similar Cantor sets to positive-measure fat Cantor sets and dynamically generated real Julia sets. Together with the product principle, they show that the Green-level mechanism applies across a variety of substantially different disconnected geometries.

\section*{Acknowledgements}

The authors used OpenAI's ChatGPT as an auxiliary tool during the
development of this manuscript, including for exploratory discussions,
bibliographic assistance, and editorial revision. All mathematical
arguments, references, and claims in the final manuscript were reviewed
and verified by the authors, who take full responsibility for the content.

\end{document}